\documentclass[a4paper]{article}

  \usepackage{amsmath,amsfonts}
  \usepackage{latexsym,amssymb}
  \usepackage[normalem]{ulem}
  \usepackage{amsthm}
  \usepackage{fullpage}
  \usepackage{tikz}
  \usepackage[T1]{fontenc}
 \usepackage[utf8]{inputenc}
  \usepackage{comment}
  \usepackage{xspace}
  \usepackage{enumerate}
  \usepackage{listings}
  \usepackage[hidelinks]{hyperref}
  \usepackage[noline,noend]{algorithm2e}
  \usetikzlibrary{decorations.pathreplacing,calc,snakes}
  \usepackage[]{authblk} 
  \usepackage{lmodern}
  \usepackage{lineno}

\renewcommand{\epsilon}{\varepsilon}

\def\A{\mathbb{A}}

\theoremstyle{plain}
\newtheorem{theorem}{Theorem}
\newtheorem{prop}[theorem]{Proposition}
\newtheorem{lem}[theorem]{Lemma}

\theoremstyle{definition}
\newtheorem{defi}[theorem]{Definition}

\theoremstyle{remark}
\newtheorem{remark}{Remark}

\DeclareMathOperator{\fact}{Fact}  
  
\begin{document}


\title{Anti-Powers in Primitive Uniform Substitutions}

\author{
Mickaël Postic

\textit{Institut Camille Jordan, Universit\'e Claude Bernard Lyon 1,  Lyon, France}

\texttt{postic@math.univ-lyon1.fr}
}

\date{}

\sloppy  

\maketitle

\begin{abstract}
In a recent work, A. Berger and C. Defant showed that if $x$ is a fixed point of a binary uniform and primitive morphism, then there exists a constant $C$ such that for all positive integers $i,k,$ beginning in position $n$ in $x$ is a $k$-anti-power with block length at most $Ck$. They ask whether this result extends to a broader class of morphic words. In this note we extend their results to fixed points of uniform primitive morphisms on arbitrary finite alphabets. Our methods make use of the recognisability of uniform primitive morphisms. This result was proved independantly by S. Garg, using a different technique \cite{Garg}.
\end{abstract}
\section{Introduction}
Since its introduction by Fici et al.~\cite{icalp}, the notion of \textit{anti-powers} has generated a lot of interest \cite{BeDe,FiPoSi,De,Gaetz,Burcroff,Antipowersinstring}. A $k$-anti-power is the concateation of $k$ blocks of same lengths that are pairwise different. In one of the latest contributions, A. Berger and C. Defant~\cite{BeDe} studied the block length of anti-powers arising in morphic words. Up to changing a little the definitions, Theorem 5 in ~\cite{BeDe} can be formulated the following way:
\begin{theorem}\label{theorem1}
If $w$ is aperiodic, fixed point of a primitive binary uniform morphism, then there is a constant $C=C(w)$ such that $\forall n,\ k \in \mathbb{N},\ w$ contains a $k$-anti-power with blocks of length at most $Ck$ beginning at its $n^{th}$ position.
\end{theorem}

They also asked to what extent are these results generalisable to a broader class of morphic words, and in particular if it was still true without the binary condition. Using the notion of recognisability first introduced by B. Moss\'e in \cite{Mosse1996}, we show that their results extend to fixed points of uniform primitive morphisms on arbitrary finite alphabets:
\begin{theorem}\label{mainthrm}
If $\sigma$ is primitive and $m$-uniform, with an aperiodic fixed point $x$, there exist a constant $C=C(\sigma)$ such that : $\forall y \in X(\sigma),\ \forall n,k \in \mathbb{N}, \ y$ contains a $k$-antipower with block length at most $Ck$ starting at position $n$.
\end{theorem}

\section{Preliminaries}
We begin by recalling some basic notions pertaining to morphisms and words.

Let us fix, for the rest of this article, a set $\mathbb{A}$ called the \textit{alphabet} whose elements are called \textit{letters}. The set of \textit{finite non empty words} $\A^{+}$ is $\bigcup\limits_{n\in \mathbb{N}}\mathbb{A}^n$, while $\mathbb{A}^*=\mathbb{A}^{+} \cup \lbrace \epsilon \rbrace$ is the set of finite words ($\epsilon$ is the \textit{empty} word). For a word $w \in \mathbb{A}^n$ we write $n=|w|$ the \textit{length} of $w$. Let $\mathbb{A}^{\mathbb{N}}=\lbrace a_1a_2\cdots | a_i\in \mathbb{A}\rbrace$ be the set of \textit{infinite words} over $\mathbb{A}$. For a finite or infinite word $w$, the notation $w_{[n,m]}$ will refer to the \textit{factor} $w_nw_{n+1}\cdots w_m$. The set of factors of $w$ will be denoted by $\fact (w)$.
A \textit{morphism} $\sigma$ over $\mathbb{A}$ is a map $\sigma :\ \mathbb{A}^* \rightarrow \mathbb{B}^*$ where $\mathbb{B}$ is another alphabet such that $\sigma(ww')=\sigma(w)\sigma(w')$. A \textit{substitution} $\sigma$ over $\mathbb{A}$ is a morphism where $\mathbb{A}=\mathbb{B}$.
A \textit{uniform} morphism $\sigma$ is a morphism of constant length over letters: $\forall\ a, b \in \mathbb{A},\ |\sigma(a)|=|\sigma(b)|$. 

\begin{remark}
In \cite{BeDe} the term morphism is used instead of substitution. We decided to stick to the term substitution, since it is commonly used in articles pertaining to recognizability. 
\end{remark}

\begin{defi}\label{primitif}
Let $\mathbb{A}=\lbrace a_1, \cdots , a_r\rbrace$. A morphism $\sigma$ over $\mathbb{A}$ is said to be \textit{primitive} if:
\[ \exists n,\ \forall i, \ \forall j,\ a_j \text{ occurs in } \sigma^n(a_i) .\]
\end{defi}
For a morphism $\sigma$, $x\in \mathbb{A}^{\mathbb{N}}$ is called a \textit{fixed point} if $x=\sigma(x)$. The \textit{shift orbit closure} $X(x)$ is the closure under the natural topology on $\mathbb{A}^{\mathbb{N}}$ of the orbit of $x$ under the shift operator $\tau : a_1a_2\cdots \rightarrow a_2a_3\cdots$. If $\sigma$ is primitive, it is easy to see that $X(x)=X(y)$ for any $x$ and $y$ fixed points of $\sigma$. Hence we can define $X(\sigma)=X(x)$ in this case.
\begin{defi}\label{reconnaissable}
A $m$-uniform primitive morphism $\sigma$ is said to be \textit{recognizable} if $\exists N \in \mathbb{N}$ such that $\forall y \in X(\sigma), \forall w\in \mathbb{A}^{+}, \sigma(y)_{[ \alpha,\alpha+|w|-1 ]}=\sigma(y)_{[ \beta,\beta+|w|-1 ]}=w$ with $|w| \geq N$ and $\alpha=0 \pmod m $ then $\beta=0 \pmod m $. $N$ is refered to as \textit{recognizability constant} of $\sigma$ in this article.
\end{defi}

\begin{remark}
Let $\sigma$ be a $m$-uniform primitive recognizable morphism and $N$ given by Definition \ref{reconnaissable}. Then $\forall y \in X(\sigma), \forall w\in \mathbb{A}^{+}, \sigma(y)_{[ \alpha,\alpha+|w|-1 ]}=\sigma(y)_{[ \beta,\beta+|w|-1 ]}=w$ with $|w| \geq N+m$ gives $\alpha=\beta \pmod m $.
\end{remark}
\begin{proof}
Denote $h=\beta -\alpha$ and $w'= \sigma(y)_{[ m\lceil \frac{\alpha}{m}\rceil,\alpha+|w|-1 ]}.$ Then $|w'| \geq N$ and $w'= \sigma(y)_{[ h+m\lceil \frac{\alpha}{m}\rceil,h+\alpha+|w|-1 ]}.$ By definition, this implies $h= 0 \pmod m$ hence $\alpha=\beta \pmod m $.
\end{proof}

\begin{remark}
If $r=|\mathbb{A}|=2$, an aperiodic word $w$ that is a fixed point of a $m$-uniform morphism $\sigma$ is uniformly recurrent if and only if $\sigma$ is primitive.
\end{remark}
\begin{proof}
Let us first suppose $\sigma$ is primitive. Without loss of generality, let us fix $w=\sigma^{\infty}(0)$. Let us denote $\mathbb{A}=\lbrace 0;1\rbrace$ with $w=\sigma^{\infty}(0)$. It is easy to see that the $n$ in Definition \ref{primitif} is 2 or 1. Let then $x$ be a factor of $w$. There exists $k$ such that $x \in \fact(\sigma^k(0))$. Every factor of $w$ of length at least $2m^{k+2}$ contains, for some $a \in \mathbb{A}$, $\sigma^{k+2}(a)$, so it contains $\sigma^{k}(0)$ hence $x$, and so $w$ is uniformly recurrent.

Let us now suppose $\sigma$ is not primitive. Since $\sigma$ is $m$-uniform, $w$ has to be eventually periodic. Indeed, if $\sigma(0)=0^m$, then $w=0^{\infty}$. If not and $\sigma(1)=1^m$, then $w$ contains arbitrary long plages of 1 ($1^{\infty}\in X(w)$), hence arbitrary long factors do not contain the factor 0. The only option left is $\sigma(0)=1^m\ \text{and } \sigma(1)=0^m$. But this leads to no fixed point ($\sigma(0)$ must start with a 0).
\end{proof}


We will now give some well-known results on substitutions and recognizability that we will need later:
\begin{theorem}[Corollaire 3.2 in \cite{Mosse1996}]\label{theorem3}
Let $\sigma$ be a primitive $m$-uniform substitution and let $x$ be aperiodic such that $\sigma(x)=x$. Then $\sigma$ is recognizable.
\end{theorem}
We will also use the following proposition from \cite{HoZa} (actually what is proved is somewhat stronger, but we only need this formulation):
\begin{prop}\label{proposition1}
If $\sigma$ is a $m$-uniform morphism, and $x$ is aperiodic with $x=\sigma(x)$ and $x_0=a$, then $\exists N_1 \in \mathbb{N}$ such that $l\geq N_1$ implies that each occurrence $\sigma^l(a)$ in $x$ is the image under $\sigma$ of an occurrence of $\sigma^{l-1}(a)$ in $x$.
\end{prop}

\section{Main Part}
The goal of this part is to prove Therorem \ref{mainthrm}.

We will first give a lemma that is easily deduced from Proposition \ref{proposition1}: 
\begin{lem}\label{corollaire1}
Let $\sigma$ be a $m$-uniform morphism, and $x$ aperiodic with $x=\sigma(x)$ and $x_0=a$. Let $N_1 \in \mathbb{N}$ be given by Proposition \ref{proposition1}. Then for every $r,\ l\in \mathbb{N}$, each occurrence $\sigma^{l+N_1+r}(a)$ in $x$ is the image under $\sigma^l$ of an occurrence of $\sigma^{N_1+r}(a)$ in $x$.
\end{lem}
\begin{proof}
By induction on $l$. Initialization is just the result of Proposition \ref{proposition1}. Let $l \in \mathbb{N}$ be fixed. Suppose the result holds for $l$ and let $x_{[ \alpha,\alpha+m^{l+1+N_1+r}-1 ]}=\sigma^{l+1+N_1+r}(a)$. Then by the recurrence hypothesis, $\alpha=m^l\alpha'$ and $x_{[ \alpha',\alpha'+m^{1+N_1+r}-1 ]}=\sigma^{1+N_1+r}(a)$. But now we can apply Proposition \ref{proposition1}: $\alpha'=m\alpha''$ and $x_{[ \alpha'',\alpha''+m^{N_1+r}-1 ]}=\sigma^{N_1+r}(a)$. So $x_{[ \alpha,\alpha+m^{l+1+N_1+r}-1 ]}=\sigma^{l+1+N_1+r}(a)$ is the image under $\sigma^{l+1}$ of an occurrence of $\sigma^{N_1+r}(a)$ in $x$.
\end{proof}

We can now prove the following lemma: 
\begin{lem}\label{lemme1}
If $\sigma$ is a primitive $m$-uniform morphism, and $x $ is aperiodic with $x =\sigma(x)$, $\exists N' \in \mathbb{N}$ such that $\forall i\in \mathbb{N},\ \sigma^i$ is recognizable with a recognizability constant less or equal to $m^iN'$. 
\end{lem}
\begin{proof}
Let then $\sigma$ be an aperiodic $m$-uniform morphism, and $x=\sigma(x)$ with $x_0=a$. Let $N_1$ be given by Proposition \ref{proposition1}, and let $y \in X(\sigma)$.

By Theorem \ref{theorem3}, $\sigma$ is recognizable. Let $N$ be a recognizability constant of $\sigma$ and let $r \in \mathbb{N}$ such that \newline $m^{N_1+r} \geq N+m $. The prefix of $x$ of length $m^{N_1+r}$ is $p=\sigma^{N_1+r}(a)$. Since $\sigma$ is primitive and uniformly recurrent, $\exists M \in \mathbb{N}$ such that every factor of $x$, hence of $y$, of length at least $M$ contains $p$. I then claim that $N'=2M$ has the required property.

Indeed, let $w,\ |w| \geq N'm^i$ be fixed. We show the following:
\[\sigma^i(y)_{[ \alpha,\alpha+|w|-1 ]}=\sigma^i(y)_{[ \beta,\beta+|w|-1 ]}=w \Rightarrow \beta=\alpha \pmod {m^i}  \quad (1).\]
Let $\alpha, \beta$ be as in (1) and $h=\beta-\alpha$. So:
\[\sigma^i(y)_{[ \alpha,\alpha+|w|-1 ]}=\sigma^i(y)_{[ \alpha+h,\alpha+h+|w|-1 ]}=w.\]
By $y\in X(\sigma)$ we get an $\alpha'$ with
\[x_{[ \alpha',\alpha'+|w|-1 ]}=x_{[ \alpha'+h,\alpha'+h+|w|-1 ]}=w. \quad (2)\]
Since $|w|\geq N'm^i=2Mm^i$, there exists $\gamma \geq 0,w' \in F(w)$ with \newline $w' =x_{ [ \gamma m^i,(\gamma+M)m^i-1]}=x_{ [ \gamma m^i+h,(\gamma+M)m^i+h-1]}$. Let $z = x_{ [ \gamma,\gamma+M-1]}$ so $w'=\sigma^i(z)$. \newline
Since $|z|=M, \ \exists \gamma',\ \gamma\leq \gamma' < \gamma'+m^{N_1+r}-1 \leq \gamma+M-1$ such that: 
\[x_{ [ \gamma',\gamma'+m^{N_1+r}-1]}=\sigma^{N_1+r}(a).\]
Applying $\sigma^i$ to $x$ gives $x_{ [ \gamma' m^i,(\gamma'+m^{N_1+r})m^i-1]}=\sigma^{N_1+r+i}(a)$, and by (2),
\[x_{ [ \gamma'm^i,(\gamma'+m^{N_1+r})m^i-1]}=x_{ [ \gamma'm^i+h,(\gamma'+m^{N_1+r})m^i+h-1]}=\sigma^{N_1+r+i}(a).\]
Using Lemma \ref{corollaire1}, this implies $h=0 \pmod{m^i}$: $x_{ [ \gamma'm^i+h,(\gamma'+m^{N_1+r})m^i+h-1]}$ is the $\sigma^i$ image of $x_{ [ \delta,\delta+m^{N_1+r}]}=\sigma^{N_1+r}(a)$, hence 
\[\gamma'm^i+h=\delta m^i \text{ and so } h = 0 \pmod {m^i}.\]
\end{proof}
\begin{proof}[Proof theorem \ref{mainthrm}]
Let $\sigma$ be aperiodic, primitive and $m$-uniform and let $C(\sigma)=(N'+1)m$ where $N'$ is the constant given by Lemma \ref{lemme1}. Let then $k,n \in \mathbb{N}$ and $y\in X(\sigma)$ be fixed. 
 
Let $i \in \mathbb{N}$ be such that $m^{i-1} \leq k < m^i$. Consider then the $k$ consecutive blocks of length $N' m^i+1$ starting at position $n$: the block number $s$ is then $y_{ [ n+s(N'm^i+1),n+(s+1)(N'm^i+1)-1]}$. We have $kC(\sigma) \geq (N'+1)m^i \geq N' m^i+1$. Moreover, using Lemma \ref{lemme1}, we get that two of these blocks, say blocks $s$ and $t$, are equal implies the difference between their starting indices is 0 modulo $m^i$: 
\begin{align}
 y_{ [ n+s(N'm^i+1),n+(s+1)(N'm^i+1)-1]}&=y_{ [ n+t(N'm^i+1),n+(t+1)(N'm^i+1)-1]}
 \\ \Rightarrow n+s(N'm^i+1)&=n+t(N'm^i+1)\pmod{m^i}
 \\ \Rightarrow s&=t\pmod{m^i}.
\end{align}
But since $k$ is smaller than $m^i$ this implies $s=t$, which completes the proof.
\end{proof}

\section{Conclusion}
Using the theory of recognizability, we have been able to improve the previous result to the class of aperiodic fixed points of primitive and $m$-uniform morphisms. The two followig examples show that the conditions  aperiodic and primitive are tight.
Let $\sigma  :
\left\{
	\begin{array}{ll}
		0  \rightarrow 01 \\
		1 \rightarrow 01 
	\end{array}
\right.$. Since $(01)^{\infty}$, which is not aperiodic, does only contain two factor of each length, it cannot contain $k$-anti-powers for $k$ greater than 2.
Let then $\sigma  :
\left\{
	\begin{array}{ll}
		0  \rightarrow 010 \\
		1 \rightarrow 111 
	\end{array}
\right.$. This substitution is not primitive, and give rise to the celebrated Cantor word $c=\sigma^{\infty}(0)$. Since $c$ contains arbitrary long plages of 1, it is clear that Theorem \ref{mainthrm} doesn't apply here. 

In the other hand, this result might be extendable to the class of recognizable substitutions, alas, I was not successful in finding an equivalent to Lemma \ref{lemme1}; this seems to be the key to extend this result. 

\bibliographystyle{alpha}
\bibliography{Biblio1}
\end{document}